\documentclass[12pt]{article}
\usepackage{epsfig,psfrag,amsmath,amssymb,latexsym}
\usepackage{amscd }
\usepackage{color}
\usepackage{amsfonts}
\usepackage{graphicx}
\usepackage{wasysym}
\newtheorem{fact}{Fact}[section]

\newtheorem{theorem}[fact]{Theorem}

\newtheorem{corollary}[fact]{Corollary}

\newtheorem{lemma}[fact]{Lemma}

\newenvironment{proof}[1][Proof]{\textbf{#1.} }{\ \rule{0.5em}{0.5em}}

\numberwithin{equation}{section}
\newcommand{\eps}{{\boldsymbol{\varepsilon}}}
\newcommand{\diag}{\operatorname{diag}}
\newcommand{\var}{{\rm Var}}

\newcommand{\ul}{{\underline{L}}}
\newcommand{\ol}{{\overline{L}}}

\newcommand{\konv}[1]{\stackrel{#1}{\longrightarrow}}
\newcommand{\connected}{\stackrel{T}{\leftrightarrow}}

\def\R{{\mathbb R}}  
\def\N{{\mathbb N}}  
\def\Z{{\mathbb Z}}  
\def\F{{\mathcal{F}}}
\def\T{{\mathcal{T}}}
\def\G{{\mathcal{G}}}  
\def\obs{{\mathcal{O}}}  

\newcommand{\ceins}{c_{4}}
\newcommand{\czwei}{c_{5}}
\newcommand{\cdrei}{c_{1}}
\newcommand{\cvier}{c_{2}}
\newcommand{\cfuenf}{c_{3}}
\newcommand{\csechs}{c_{6}}
\newcommand{\csieben}{c_{7}}
\newcommand{\cacht}{c_{8}}

\begin{document}
\thispagestyle{empty}

\begin{center}
{\LARGE A comparison of the nonlinear sigma model \\ with general pinning and 
pinning at one point
}\\[3mm]
{\large Margherita Disertori\footnote{Institute for Applied Mathematics
\& Hausdorff Center for Mathematics, 
University of Bonn,
Endenicher Allee 60,
D-53115 Bonn, Germany.
E-mail: disertori@iam.uni-bonn.de}
\hspace{1cm} 
Franz Merkl \footnote{Mathematical Institute, University of Munich,
Theresienstr.\ 39,
D-80333 Munich,
Germany.
E-mail: merkl@math.lmu.de
}
\hspace{1cm} 
Silke W.W.\ Rolles\footnote{Zentrum Mathematik, Bereich M5,
Technische Universit{\"{a}}t M{\"{u}}nchen,
D-85747 Garching bei M{\"{u}}nchen,
Germany.
E-mail: srolles@ma.tum.de}
\\[3mm]
{\small \today}}\\[3mm]
\end{center}

\begin{abstract}
We study the nonlinear supersymmetric hyperbolic sigma model introduced 
by Zirnbauer in 1991. This model can be related to the mixing measure of a
vertex-reinforced jump process. We prove that the two-point correlation 
function has a probabilistic interpretation in terms of connectivity in 
rooted random spanning forests. Using this interpretation, we dominate the 
two-point correlation function for general pinning, e.g.\
for uniform pinning, with the corresponding correlation function with pinning at 
one point. The result holds for a general finite graph, asymptotically as the
strength of the pinning converges to zero. Specializing this to general ladder graphs, 
we deduce 
in the same asymptotic regime exponential decay of correlations for general pinning. 
\end{abstract}

\section{Introduction}

\paragraph{History of the model.} 
The nonlinear sigma model that we consider here was introduced by Zirnbauer in 
\cite{zirnbauer-91} as a toy model inspired by random matrices in the context of disordered 
materials. It can be seen as a 
statistical mechanical model where spins are replaced by vectors with both real and 
Grassmann components. We associate with each point two real and two Grassmann variables
parametrizing a supersymmetric extension of the hyperbolic plane. Therefore the model is 
often denoted by $H^{2|2}$. In dimension three and higher, a phase transition between a 
localized (disordered) and an extended (ordered) phase was proved by Disertori, Spencer, 
and Zirnbauer in 
\cite{disertori-spencer-zirnbauer2010} and \cite{disertori-spencer2010}. 

After integrating out the Grassmann variables in the nonlinear sigma model, the corresponding 
marginal is a probability measure. It was shown by Sabot and Tarr{\`e}s in 
\cite{sabot-tarres2012} that this marginal has an interpretation as a mixing measure 
for a vertex-reinforced jump process and can also be related to linearly edge-reinforced
random walk. Exploiting the former relation, the results in
\cite{disertori-spencer-zirnbauer2010} and \cite{disertori-spencer2010} were used by 
Sabot and Tarr{\`e}s in \cite{sabot-tarres2012} to deduce recurrence for vertex-reinforced 
jump processes in all dimensions for small initial weights and transience in dimensions 
$d\ge 3$ for large initial weights. For linearly edge-reinforced random walks, 
Sabot and Tarr{\`e}s proved recurrence in all dimensions for small initial weights. 
An alternative proof, without using the connection to $H^{2|2}$, was given by Angel, Crawford, 
and Kozma in \cite{angel-crawford-kozma}. In dimensions $d\ge 3$, Disertori, Sabot, and 
Tarr{\`e}s showed in \cite{disertori-sabot-tarres2014} transience for linearly 
edge-reinforced random walks for large initial weights using techniques similar to the one 
used in \cite{disertori-spencer-zirnbauer2010}. In \cite{disertori-merkl-rolles2014}, 
we proved recurrence for vertex-reinforced jump processes on general ladder graphs
with arbitrary constant initial weights using the connection to $H^{2|2}$ just mentioned 
and a transfer operator method applied to $H^{2|2}$. 

\paragraph{Aim of this paper.}
Without a regularization the model $H^{2|2}$ 
is ill-defined. On a lattice, the most natural choice is to introduce a 
translationally invariant regularization. This is equivalent to introduce  a constant 
``magnetic field'' $\eps$ in the corresponding statistical mechanics model. 
This magnetic field at point $j$ can be interpreted as a ``pinning'', forcing the 
spin at point $j$ to remain near a certain value. Then, a constant magnetic field
can be seen as uniform pinning. Another possibility is to take an 
inhomogeneous magnetic field, the easiest choice being pinning at a single point. 

In this paper, we consider the model on a finite connected undirected graph $\G$, 
rather than only lattices. The pinning can also be seen as a modification of the 
underlying graph as follows. We augment $\G$ 
by an additional vertex $\rho$ in two different ways. In the case of uniform pinning, 
$\rho$ is connected to every other vertex. In the case of pinning at one point, $\rho$ 
is only connected to a single vertex in $\G$. When $\G$ is a lattice or a ladder graph, 
the first graph 
has a nonlocal structure since the graph distance between any two vertices is bounded
by $2$, whereas the second graph remains local. 

In the case of ladder graphs, the local structure for pinning at one point allowed us 
to prove exponential decay of correlations for arbitrary inverse temperature $\beta$;
see \cite{disertori-merkl-rolles2014}. However, due to the nonlocal structure of the 
augmented graph, a similar method is not directly applicable for uniform pinning. 

The aim of this paper is to bound the expectation of the Green's function in the 
case of uniform (or general) pinning with the corresponding Green's function for 
pinning at one point, asymptotically as $\eps\to 0$, for any inverse temperature 
$\beta>0$. This result holds for general finite graphs. Specializing it down to 
ladder graphs, it allows us to transfer known bounds 
for pinning at one point to the case of the Green's function for uniform pinning. 

\paragraph{How this paper is organized.}
In Section \ref{sec:model-results}, the model is formally defined and the results are stated. 
In Subsection \ref{sec:interpretation-Gxy}, we relate the Green's function 
with a probability concerning certain random spanning trees. Subsection 
\ref{sec:comparison-pinnings} contains the proof of the comparison between the different 
pinnings. The model $H^{2|2}$ with pinning at one point has a product structure when
passing to gradient variables, that the model $H^{2|2}$ with uniform pinning does not 
exhibit. This product structure is explained in the appendix.  

\section{Model and results}
\label{sec:model-results}

\subsection{Formal definition}
\label{subsec:formal-def}

Let $\G=(V,E)$ be a finite connected graph with vertex set 
$V$ and edge set $E$, consisting of undirected edges $i\sim j$. We extend $\G$ by 
adding an extra vertex $\rho$ 
and edges $i\sim\rho$ connecting it to every other vertex $i$. The extended graph is 
denoted by $\G_\rho:=(V_\rho,E_\rho)$ with $V_\rho:=V\cup\{\rho\}$ and 
$E_\rho:=E\cup\{(i\sim\rho):i\in V\}$. 

We attach to every edge $i\sim j$ of $\G_\rho$ an edge weight $\beta_{ij}=\beta_{ji}>0$.
In particular, for $i\in V$, the edge $i\sim\rho$ gets the weight 
$\beta_{i\rho}=:\eps_i\ge 0$. 
We assume that $\eps_i>0$ for at least one vertex $i\in V$. 
To every vertex $i\in V$, we associate two real variables $t_i$ and 
$s_i$, and abbreviate $t:=(t_i)_{i\in V}$ and $s:=(s_i)_{i\in V}$. Furthermore, we set 
$t_\rho:=0$ and $s_\rho:=0$. For $i,j\in V_\rho$, we define
\begin{align}
B_{ij}(t,s) &:= \cosh(t_i - t_j)+\tfrac12 (s_i-s_j)^2 e^{t_i+t_j}.
\end{align} 
In particular, for $i\in V$, we have 
\begin{align}
B_{i\rho}(t,s) &= \cosh(t_i)+\tfrac12 s_i^2 e^{t_i}. 
\end{align} 
In the following, we study an equivalent version of the model $H^{2|2}$, where the 
contribution from Grassmann variables is replaced by   a sum over  
spanning trees $T$ of $\G_\rho$. 
Let $\T$ denote the set of spanning trees of $\G_\rho$. We identify every tree
with its edge set. 
The spanning trees $T$ of $\G_\rho$ are in a natural one-to-one correspondence with 
rooted spanning forests of $\G$ as follows. Given $T\in\T$, the corresponding spanning 
forest has the edge set 
\begin{align}
\label{eq:def-forest}
F(T):=T\cap E
\end{align} 
and the set of roots 
\begin{align}
\label{eq:def-roots}
R(T):=\{i\in V:\;(i\sim\rho)\in T\}. 
\end{align}
Using this notation, we have for $t\in\R^V$ and $T\in\T$
\begin{align}
\label{eq:tree-contribution}
\prod_{(i\sim j)\in T}\beta_{ij}e^{t_i+t_j}
=\prod_{(i\sim j)\in F(T)}\beta_{ij}e^{t_i+t_j}\prod_{i\in R(T)}\eps_ie^{t_i}. 
\end{align}
In this representation, $H^{2|2}$ is described by the following probability measure on 
\mbox{$\R^V\!\!\times\!\R^V\!\!\times\!\T$}
\begin{align}
& \mu^\eps (dt\, ds\, dT) 
:=  
\prod_{j\in V} \frac{dt_jds_je^{-t_j}}{2\pi }\, dT
\prod_{(i\sim j)\in E_\rho} e^{-\beta_{ij}(B_{ij}(t,s)-1)}
\prod_{(i\sim j)\in T} \beta_{ij}e^{t_i+t_j}
\label{eq:mu-eps}\\
= & 
\prod_{j\in V} \frac{dt_jds_je^{-t_j}}{2\pi }\, dT
\prod_{(i\sim j)\in E} e^{-\beta_{ij}(B_{ij}(t,s)-1)}
\prod_{i\in V} e^{-\eps_i (B_{i\rho}(t,s)-1)}
\prod_{(i\sim j)\in F(T)}\beta_{ij}e^{t_i+t_j}\prod_{i\in R(T)}\eps_ie^{t_i},
\nonumber
\end{align}
where $dt_j$ and $ds_j$, $j\in V$, denote the Lebesgue measure on $\R$ 
and $dT$ is the counting measure on $\T$. As is shown in 
\cite{disertori-spencer-zirnbauer2010}, using supersymmetry and the equivalent description 
in terms of Grassmann variables, $\mu^\eps$ is a probability measure. 

For convenience, we suppress the dependence of $\mu^\eps$ on $\G$ and $\beta$ in the 
notation. 
The symbols $t$, $s$, and $T$ are used in two 
slightly different ways: On the one hand, they
denote the canonical projections on $\R^V\times\R^V\times\T$. On the other
hand, the same symbols are used as integration (or summation) variables. 

Consider the matrix $A(t)\in\R^{V\times V}$ defined by 
\begin{align}
A(t)_{ij}:=\left\{\begin{array}{ll} 
- \beta_{ij} e^{t_i+t_j} & \text{for }(i\sim j)\in E, \\[1mm]
\sum\limits_{\substack{k\in V:\\ (k\sim j)\in E}} \beta_{kj} e^{t_k+t_j} & \text{for }i=j, \\[1mm]
0 & \text{else.}
\end{array}\right.
\end{align}
Let $A^\rho(t)$ denote the analog of $A(t)$ when the underlying graph is 
$\G_\rho$ instead of $\G$ and let $\hat\eps$ be the following diagonal 
matrix: 
\begin{align}
\hat\eps:=\diag(\eps_i e^{t_i},i\in V).
\end{align} 
Deleting the row and the column indexed by $\rho$ from $A^\rho(t)$, 
we get the matrix $A^\rho(t)_{\rho^c\rho^c}=A(t)+\hat\eps(t)$. 
Hence, by the well-known matrix tree theorem (see \cite{Abdesselam} for a 
generalized version), 
\begin{align}
\label{eq:matrix-tree}
\det(A(t)+\hat\eps)=\det \left(A^\rho(t)_{\rho^c\rho^c}\right)
=\sum_{T\in\T}\prod_{(i\sim j)\in T}\beta_{ij} e^{t_i+t_j}. 
\end{align}
Consequently, the $(t,s)$-marginal of $\mu^\eps$ is precisely the 
supersymmetric sigma model studied e.g.\ in \cite{disertori-spencer2010}. 

\subsection{Results}
The proofs of all results stated here are given in Section \ref{sec:proof}.
Let $x,y\in V$ be two different vertices and $\pi_i\ge 0$, $i\in V$, be 
fixed numbers with $\pi_x,\pi_y>0$. 
Given $\epsilon>0$, we set $\eps=(\eps_i)_{i\in V}=(\pi_i\epsilon)_{i\in V}$.
We are interested in the following Green's function
\begin{align}
\label{def:greensfn}
G_{xy}^\eps:=e^{t_x+t_y}(A(t)+\hat\eps(t))^{-1}_{xy}.
\end{align}
It has the following probabilistic interpretation in terms of random weighted spanning
trees: 

\begin{lemma}
\label{le:probabilistic-interpretation-Gxy}
There is a version $P_{\beta,t}$ of the conditional law of $\mu^\eps$ given $t$ 
that fulfills 
\begin{align}
\label{eq:formula-P-beta-t}
P_{\beta,t}(T=S)=
\frac{\prod_{e\in S}\beta_e(t)}{\sum_{S'\in\T}\prod_{e'\in S'}\beta_{e'}(t)},
\quad\text{for } S\in\T, 
\end{align}
with $\beta_e(t):=\beta_{ij}e^{t_i+t_j}$ for $e=(i\sim j)$. 
Writing $x\connected y$ if $x$ and $y$ are connected
in the spanning tree $T$ through a path which does not use $\rho$, we have
\begin{align}
\label{eq:formula-Gxy}
G_{xy}^\eps=\frac{e^{t_x+t_y}}{\eps_x e^{t_x}+\eps_y e^{t_y}}
P_{\beta,t}( \{ (x\sim\rho)\in T \text{ and }x\connected y\}
\cup \{ (y\sim\rho)\in T \text{ and }x\connected y\} ) .
\end{align}
\end{lemma}

We define
\begin{align}
\label{eq:def-obs}
\obs_{xy}^\pi:=\frac{e^{t_x+t_y}}{\pi_x e^{t_x}+\pi_y e^{t_y}}
1_{ \{ (x\sim\rho)\in T \text{ and }x\connected y\}}. 
\end{align}
As a consequence of Lemma \ref{le:probabilistic-interpretation-Gxy}, one has 
\begin{align}
\label{eq:G-xy}
G_{xy}^\eps=\epsilon^{-1}E_{\mu^\eps}[\obs_{xy}^\pi+\obs_{yx}^\pi|t]
=E_{\mu^\eps}[\obs_{xy}^\eps+\obs_{yx}^\eps|t].
\end{align}
The expression $E_{\mu^\eps}[\cdot|t]$ stands for the conditional expectation
given $t$. 

We denote by $\eps_x\delta_x\in\R^V$ the vector with coordinate
$\eps_x$ at $x$ and coordinates $0$ at all other locations. Our main theorem can 
now be phrased as follows. 

\begin{theorem}
\label{thm:comparison}
Let $\pi_i\ge 0$, $i\in V$, be fixed numbers with $\pi_x,\pi_y>0$. 
We have the following asymptotic
comparison between the supersymmetric sigma model with arbitrary 
pinning $\eps=(\eps_i)_{i\in V}=(\pi_i\epsilon)_{i\in V}$ and with pinning at one point:
\begin{align}
0<\lim_{\epsilon\downarrow 0} \epsilon E_{\mu^\eps}[G_{xy}^\eps]
\le \lim_{\epsilon\downarrow 0} \left( E_{\mu^{\eps_x\delta_x}}\left[\obs_{xy}^\pi\right]
+ E_{\mu^{\eps_y\delta_y}}\left[\obs_{yx}^\pi\right] \right)<\infty
\label{eq:comp}\end{align}
In particular, the two limits exist. 
\end{theorem}

\paragraph{Remark.}
It is important that in  \eqref{eq:comp} we keep the random variable $\obs_{xy}^\pi$
with the original $\pi$. Indeed, if we replaced  $\obs_{xy}^\pi$ by  
$\obs_{xy}^{\pi_{x}\delta_{x}}$ we would get 
\begin{equation}
 E_{\mu^{\eps_x\delta_x}}\left[\obs_{xy}^{\pi_{x}\delta_{x}}\right]= 
E_{\mu^{\eps_x\delta_x}}\left[\frac{e^{t_{y}}}{\pi_{x}}\right]=\frac{1}{\pi_{x}}
\end{equation}
by formula \eqref{eq:expectation-2} in the appendix. 
This bound would give no information on eventual
decay in $|y-x|$.

\paragraph{Ladder graphs.} In the special case of quasi-one-dimensional graphs
this theorem implies exponential decay of the expectation of $G_{xy}^\eps$. 
More precisely, consider a finite undirected graph $G_0=(V_0,E_0)$ with vertex set $V_0$ and 
edge set $E_0$. Let $\G=(V,E)$ be the ``ladder'' built of copies $G_n=(V_n,E_n)$,
$n\in\Z$, of $G_0$. More precisely, we take the vertex sets $V:=\Z\times V_0$ and 
$V_n:=\{n\}\times V_0$, where the copy at level $0$ is identified with $G_0$. 
The edge set $E$ consists of ``vertical'' edges 
$e_n:=((n,v)\sim(n,v'))$ with $n\in\Z$ and $e=(v\sim v')\in E_0$ and ``horizontal'' edges
$v_{n+1/2}:=((n,v)\sim(n+1,v))$ with $n\in\Z$ and $v\in V_0$. 
For $\ul,\ol\in\N$, we set $L:=(-\ul,\ol)$ and denote by $\G_L$ the subgraph of 
$\G$ consisting of the vertex set $V_L:=\{-\ul,\ldots,\ol\}\times V_0$ and 
the edge set $E_L$ containing all edges $e\in E$ connecting two vertices in $V_L$.  
We associate with every edge $e\in E$ a weight $\beta_e>0$. We assume that 
the family of weights $\beta:=(\beta_e)_{e\in E}$ is translation invariant 
in the sense $\beta_{e_n}=\beta_{e_0}$ and $\beta_{v_{n+1/2}}=\beta_{v_{1/2}}$ for all
$n\in\Z$, $e\in E_0$, and $v\in V_0$. For $x=(n,v)$ and $y=(m,w)$ in $V$, their
horizontal distance is defined by $|x-y|:=|n-m|$. 
Let $\mu^\eps_L$ denote the distribution of the supersymmetric sigma model 
on the graph $\G_L$ with pinning $\eps=(\eps_i)_{i\in V_L}=(\epsilon\pi_i)_{i\in V_L}$, 
where $\pi_i\ge 0$, $i\in V_L$, with at least one $\pi_i>0$.

\begin{corollary}
\label{cor:exp-decay}
There exist constants $\cdrei,\cvier>0$ depending only on $G_0$ and $\beta$
such that for any $L$, any two different vertices $x,y\in V_L$ 
with $\pi_x,\pi_y>0$, and $\cfuenf(\pi):=\min \{\pi_x,\pi_y\}^{-1}$ one has 
\begin{align}
\label{eq:exp-decay-ladder}
0<\lim_{\epsilon\downarrow 0} \epsilon E_{\mu^\eps_L}[G_{xy}^\eps]
\le \cdrei\cfuenf e^{-\cvier|x-y|}. 
\end{align}
\end{corollary}

For sufficiently small $\beta>0$, the methods from \cite{disertori-spencer2010} can be 
used to prove a version of \eqref{eq:exp-decay-ladder} not only for the limit as 
$\epsilon\downarrow 0$, but also for given $\epsilon>0$ small enough. In contrast, our result 
holds for arbitrary $\beta>0$, but only asymptotically for $\epsilon\downarrow 0$. 

\section{Proofs}
\label{sec:proof}

\subsection{Probabilistic interpretation of $G_{xy}^\eps$}
\label{sec:interpretation-Gxy}

Let $\G$ and $\G_\rho$ be the graphs introduced in Section \ref{subsec:formal-def}. 
Without loss of generality, we assume $V=\{1,\ldots,n\}$ throughout this subsection. 
Recall that $\T$ denotes the set of spanning trees of $\G_\rho$.
We will use
the definitions \eqref{eq:def-forest} of the forest $F(T)$ and 
\eqref{eq:def-roots} of the set of roots $R(T)$. Finally, recall that $x\connected y$  iff
 $x$ and $y$ are connected in $F(T)$.
 For any $x,y\in V$, let 
\begin{align}
\T_{xy}:=\{ T\in\T:\, x\in R(T) \mbox{ and } x\connected y\}. 
\end{align}
For the proof of Lemma \ref{le:probabilistic-interpretation-Gxy}, we need the following
result. 

\begin{lemma}
\label{le:abdesselam}
Consider $x,y\in V$ and a real symmetric matrix $M\in\R^{V\times V}$ with $M_{ij}=0$ 
whenever $i\neq j$ and there is no edge between $i$ and $j$. Then, 
the determinant of the minor of 
$M$ obtained by taking away the column $x$ and the row $y$ is given by 
\begin{align}
\label{eq:thm1-abdesselam}
\det M_{y^cx^c} = (-1)^{x+y}\sum_{T\in\T_{xy}}
\left[\prod_{j\in R(T)\setminus\{x\}}\left(\sum_{i\in V}M_{ij}\right)\right]\prod_{(i\sim j)\in F(T)}(-M_{ij}).
\end{align}
\end{lemma}
\begin{proof}
In the case that $\G$ is the {\it complete} undirected graph with vertex set $V$, 
this is the special case of Theorem 1 of Abdesselam's article 
\cite{Abdesselam} for a {\it symmetric} matrix $M$ when the index sets $I$ and $J$
are replaced by {\it singletons} $\{y\},\{x\}\subseteq\{1,\ldots,n\}$, respectively. 
Note that the sign $\eps(\F)$ appearing in Abdesselam's formula equals 1 in our 
special case because $I=\{y\}$ and $J=\{x\}$ are singletons. Since $M_{ij}=0$ whenever 
$i\neq j$ and there is no edge between $i$ and $j$, formula~\eqref{eq:thm1-abdesselam}
remains literally true if we replace the complete graph by the given graph $\G$. 
\end{proof}

\medskip\noindent\begin{proof}[Proof of Lemma \ref{le:probabilistic-interpretation-Gxy}]
The tree dependent part of the density of $\mu^\eps$ in 
\eqref{eq:mu-eps} is given by $\prod_{e\in T}\beta_e(t)$; note that $s$ and $T$ are 
conditionally independent given $t$. Consequently, 
formula \eqref{eq:formula-P-beta-t} describes indeed the law of $T$ conditional on $t$.

By definition \eqref{def:greensfn}, it holds 
\begin{align}
\label{eq:Gxy-as-fraction}
e^{-(t_x+t_y)}G_{xy}^\eps=(A(t)+\hat\eps(t))^{-1}_{xy}
=(-1)^{x+y}\frac{\det(A(t)+\hat\eps(t))_{y^cx^c}}{\det(A(t)+\hat\eps(t))}. 
\end{align}
For the numerator, we use 
Lemma \ref{le:abdesselam} with $M=A(t)+\hat\eps$.
Clearly, all columns of $A(t)$ sum up to 0. Consequently, for all $j\in V$, one has 
\begin{align}
\sum_{i\in V} M_{ij}=\sum_{i\in V} (A_{ij}(t)+\hat\eps_{ij})=\hat\eps_{jj}=\eps_je^{t_j}. 
\end{align}
Furthermore, for $(i\sim j)\in F(T)$, $T\in\T$, one has $-M_{ij}=\beta_{ij}e^{t_i+t_j}$.
Using this, formula \eqref{eq:thm1-abdesselam} yields 
\begin{align}
(-1)^{x+y}\det(A(t)+\hat\eps(t))_{y^cx^c} = \sum_{T\in\T_{xy}}
\left[\prod_{j\in R(T)\setminus\{x\}}\eps_je^{t_j}\right]\prod_{(i\sim j)\in F(T)}\beta_{ij}e^{t_i+t_j}. 
\end{align}
Multiplying this equation by $\eps_x e^{t_x}$ and using formula \eqref{eq:tree-contribution}, 
we get 
\begin{align}
(-1)^{x+y}\eps_x e^{t_x}\det(A(t)+\hat\eps(t))_{y^cx^c} 
=&  \sum_{T\in\T_{xy}}
\left[\prod_{j\in R(T)}\eps_je^{t_j}\right]\prod_{(i\sim j)\in F(T)}\beta_{ij}e^{t_i+t_j}\cr
= &\sum_{T\in\T_{xy}}\prod_{(i\sim j)\in T}\beta_{ij}e^{t_i+t_j}.
\end{align}
Hence, using \eqref{eq:matrix-tree} for the denominator, we obtain 
\begin{align}
\eps_x e^{t_x}(A(t)+\hat\eps(t))_{xy}^{-1} 
= & (-1)^{x+y}\eps_x e^{t_x}
\frac{\det(A(t)+\hat\eps(t))_{y^cx^c}}{\det(A(t)+\hat\eps(t))}\cr
= & \frac{\sum_{T\in\T_{xy}}
\prod_{(i\sim j)\in T}\beta_{ij}e^{t_i+t_j}}{\sum_{T\in\T}
\prod_{(i\sim j)\in T}\beta_{ij}e^{t_i+t_j}} \cr
=&  P_{\beta,t}(T\in\T_{xy})=P_{\beta,t}( (x\sim\rho)\in T \text{ and }x\connected y ).
\label{eq:green1}
\end{align}
Exchanging $x$ and $y$ and using the symmetry of $A(t)+\hat\eps(t)$, we get 
\begin{align}
\eps_y e^{t_y}(A(t)+\hat\eps(t))_{xy}^{-1} 
=&  P_{\beta,t}(T\in\T_{yx})=P_{\beta,t}( (y\sim\rho)\in T \text{ and }x\connected y ).
\label{eq:green2}
\end{align}
Since $x\neq y$, the sets $\T_{xy}$ and $\T_{yx}$ are disjoint. Hence, 
$P_{\beta,t}(T\in\T_{xy})+P_{\beta,t}(T\in\T_{yx})=P_{\beta,t}(T\in\T_{xy}\cup\T_{yx})$. 
Finally, we add \eqref{eq:green1} and \eqref{eq:green2} and insert them into 
\eqref{eq:Gxy-as-fraction} to obtain the claim \eqref{eq:formula-Gxy}. 
\end{proof}

\subsection{Comparing different pinnings}
\label{sec:comparison-pinnings}

We write the random variable $\obs_{xy}^\pi$ defined in \eqref{eq:def-obs} as a sum:
\begin{align}
\label{eq:obs-as-sum}
\obs_{xy}^\pi=\obs_{xy}^\pi1_{\{R(T)=\{x\}\}}+\obs_{xy}^\pi1_{\{|R(T)|>1\}}. 
\end{align}
Note that we have $\obs_{xy}^\pi1_{\{|R(T)|=1\}}=\obs_{xy}^\pi1_{\{R(T)=\{x\}\}}$
because $\obs_{xy}^\pi$ contains the indicator function of the event 
$\{(x\sim\rho)\in T\}$.

The proof of Theorem \ref{thm:comparison} is based on Lemmas \ref{le:contribution-one-root}
and \ref{le:contribution-more-than-one-root}, below,
dealing with the first and second summand in \eqref{eq:obs-as-sum}, respectively. 
Surprisingly, the main mass contributing to the expectation of 
the first and second summand in \eqref{eq:obs-as-sum} comes from 
quite different locations; see also the explanations following formula 
\eqref{eq:def-change-of-variables-t}, below. On the one hand, values $(t,s)$ with 
$t_j\approx-\log\epsilon$
carry most of the mass for the expectation of $\obs_{xy}^\pi1_{\{R(T)=\{x\}\}}$. 
On the other hand, the main contribution to the expectation of 
$\obs_{xy}^\pi1_{\{|R(T)|>1\}}$ comes from values $(t,s)$ with $t_j\approx+\log\epsilon$. 
Thus, for small $\epsilon>0$, in the two expectations the main masses sit 
at opposite ends. 
As we shall see, the main contribution comes from the term with 
precisely one root at $x$. We examine this contribution first.

\begin{lemma}[Contribution of one root]
\label{le:contribution-one-root}
For all $\epsilon>0$, 
\begin{align}
\label{eq:claim-fixed-eps}
E_{\mu^\eps}[\obs_{xy}^\pi 1_{\{R(T)=\{x\}\}}]
\le & e^{\sum_{i\in V\setminus\{x\}}\eps_i}E_{\mu^{\eps_x\delta_x}}\left[\obs_{xy}^\pi\right]. 
\end{align}
Furthermore, 
\begin{align}
\label{eq:claim-limit-eps}
0< & \lim_{\epsilon\downarrow 0} E_{\mu^\eps}[\obs_{xy}^\pi 1_{\{R(T)=\{x\}\}}]
= \lim_{\epsilon\downarrow 0} E_{\mu^{\eps_x\delta_x}}\left[
\obs_{xy}^\pi \sqrt{\tfrac{\pi_x e^{t_x}}{\sum_{i\in V}\pi_i e^{t_i}}}
\prod_{i\in V\setminus\{x\}} e^{-\frac12\epsilon\pi_i e^{t_i}} \right] \cr
\le & \lim_{\epsilon\downarrow 0} 
E_{\mu^{\eps_x\delta_x}}\left[\obs_{xy}^\pi\right]\le \min \left\{\frac{1}{\pi_x}, \frac{1}{\pi_y} \right\}.
\end{align}
In particular all the displayed limits exist. 
\end{lemma}

The proof will be a consequence of a more general result, where the family $\eps$ 
is replaced by another family $\epsilon a$ and the random variable 
$\obs_{xy}^\pi1_{\{R(T)=\{x\}\}}$ is multiplied by an additional factor $\chi$. 

Take $a=(a_i)_{i\in V}$ with $a_i\ge 0$ for all $i$ and $a_x>0$ and 
an additional density function $\chi:\R^V\to(0,1]$. We will study the 
expectation 
\begin{align}
E_{\mu^{\epsilon a}}[\chi(t+\log\epsilon)\obs_{xy}^\pi1_{\{R(T)=\{x\}\}}], 
\end{align}
where we abbreviate $t+\log\epsilon:=(t_i+\log\epsilon)_{i\in V}$. 
The main contribution to the expectation coming from values of the $t_i$'s 
close to $-\log\epsilon$ motivates us to shift the $t_i$'s by $\log\epsilon$.
For this purpose, we introduce new variables: 
\begin{align}
\label{eq:new-variables}
t_i':=t_i+\log\epsilon,\quad
s_i':= \epsilon^{-1}(s_i-s_x)\quad\text{for all }i\in V. 
\end{align}
With this definition, $s_x'=0$. Therefore, we will use as new integration variables 
$(t_i')_{i\in V}$, $s_x$, and $(s_i')_{i\in V\setminus\{x\}}$. 
For any fixed configuration $t',s'$ of these new variables, we consider an auxiliary 
random variable $S'$ on some probability space, 
taking values $s_i'$ with probabilities 
\begin{align}
\label{eq:def-rv-S}
P_{a,t',s'}(S'=s_i'):=\frac{a_ie^{t_i'}}{z_{a,t'}}, \quad i\in V, 
\end{align}
where 
\begin{align}
z_{a,t'}:=\sum_{i\in V}a_ie^{t_i'}
\end{align} 
is the normalizing constant.
We denote by $E_{a,t', s'}$ and $\var_{a,t', s'}$ the corresponding expectation and variance
operators, respectively. 
In order to have a compact notation, we abbreviate in the following 
\begin{align}
\label{eq:def-kappa}
\kappa(dt'd s')
:= &
\prod_{j\in V} \frac{dt_j'\, e^{-t_j'}}{2\pi}\,
\prod_{j\in V\setminus\{x\}} d s_j'
\prod_{(i\sim j)\in E} e^{-\beta_{ij}(B_{ij}(t', s')-1)} 
\cr
& \cdot\sum_{T\in\T} \prod_{(i\sim j)\in F(T)} \beta_{ij}e^{t_i'+t_j'}\,
e^{t_x'} \frac{e^{t_x'+t_y'}}{\pi_x e^{t_x'}+\pi_y e^{t_y'}}
1_{ \{ R(T)=\{x\}\text{ and }x\connected y\}}. 
\end{align}
Note that $\kappa$ depends also on the \textit{fixed} quantities $\pi$, 
$(\beta_{ij})_{(i\sim j)\in E}$, and on the vertices $x,y$, although this is 
not displayed. We remind that $s_x'=0$ by construction.

\begin{lemma}
\label{le:exp-chi}
With all the definitions above, we have 
\begin{align}
& e^{-\epsilon\sum_{i\in V} a_i} E_{\mu^{\epsilon a}}[\chi(t+\log\epsilon)\obs_{xy}^\pi1_{\{R(T)=\{x\}\}}] \cr
= & \int\limits_{\R^V\times\R^{V\setminus\{x\}}} \kappa(dt'd s')\,
\sqrt{\tfrac{2\pi}{z_{a,t'}}}
e^{-\frac{z_{a,t'}}{2}\epsilon^2\var_{a,t', s'}( S')}
\prod_{i\in V} e^{-\frac12 a_i ( e^{t_i'}+\epsilon^2 e^{-t_i'})} 
\cdot\, a_x\chi(t')\cr
\uparrow_{\epsilon\downarrow 0}&
\int\limits_{\R^V\times\R^{V\setminus\{x\}}} \kappa(dt'd s')
\sqrt{\tfrac{2\pi}{z_{a,t'}}}\prod_{i\in V} e^{-\frac12 a_i e^{t_i'}} 
\cdot\, a_x\chi(t')
>0.
\label{eq:limit-as-eps-to-zero}
\end{align}
\end{lemma}
\begin{proof}
Note that
on the event $\{ R(T)=\{x\}\}$, the root contribution in \eqref{eq:tree-contribution}
is given by $\epsilon a_xe^{t_x}$. Using \eqref{eq:mu-eps}, we get 
\begin{align}
&\hspace{-1cm} E_{\mu^{\epsilon a}}[\chi(t+\log\epsilon)\obs_{xy}^\pi 1_{\{R(T)=\{x\}\}}] \cr
=  & \sum_{T\in\T} \int_{(\R^V)^2} \prod_{j\in V} \frac{dt_jds_je^{-t_j}}{2\pi }
\prod_{(i\sim j)\in E} e^{-\beta_{ij}(B_{ij}(t,s)-1)} 
\prod_{i\in V} e^{-\epsilon a_i (B_{i\rho}(t,s)-1)} 
\prod_{(i\sim j)\in F(T)} \beta_{ij}e^{t_i+t_j}\cr
& \cdot\,\epsilon a_xe^{t_x}
\chi(t+\log\epsilon)\frac{e^{t_x+t_y}}{\pi_x e^{t_x}+\pi_y e^{t_y}}
1_{ \{ R(T)=\{x\}\text{ and }x\connected y\}}.
\end{align}
Changing variables according to \eqref{eq:new-variables}, we get 
$(s_i-s_j)^2e^{t_i+t_j}=(s_i'-s_j')^2e^{t_i'+t_j'}$ and  
\begin{align}
& E_{\mu^{\epsilon a}}[\chi(t+\log\epsilon)\obs_{xy}^\pi 1_{\{R(T)=\{x\}\}}] \cr
= & \sum_{T\in\T} \int_{\R^V} \prod_{j\in V} 
\frac{dt_j'\,\epsilon e^{-t_j'}}{2\pi}\,
\int_{\R}ds_x \int_{\R^{V\setminus\{x\}}}
\prod_{j\in V\setminus\{x\}} \epsilon\, d s_j' 
\prod_{(i\sim j)\in E} e^{-\beta_{ij}(B_{ij}(t', s')-1)} \cr 
&\cdot\prod_{i\in V} 
e^{-\frac12a_i ( e^{t_i'}+\epsilon^2 e^{-t_i'} -2\epsilon 
+ (\epsilon  s_i'+s_x)^2 e^{t_i'} )} 
\prod_{(i\sim j)\in F(T)} \beta_{ij}\epsilon^{-2}e^{t_i'+t_j'}\,\cdot\, a_xe^{t_x'} 
\cr
& \cdot\chi(t')\frac{\epsilon^{-1} e^{t_x'+t_y'}}{\pi_x e^{t_x'}+\pi_y e^{t_y'}}
1_{ \{ R(T)=\{x\}\text{ and }x\connected y\}}\cr
= &e^{\epsilon\sum_{i\in V} a_i} \int_{\R} ds_x \int_{\R^V\times\R^{V\setminus\{x\}}}
\kappa(dt'd s')
\prod_{i\in V} e^{-\frac12 a_i ( e^{t_i'}+\epsilon^2 e^{-t_i'} 
+ (\epsilon  s_i'+s_x)^2 e^{t_i'} )} 
\cdot\,a_x \chi(t').
\label{eq:expectation-Oxy-one-root}
\end{align}
For counting powers of $\epsilon$ in the last equality, we have used that $F(T)$ is a
 spanning tree of $\G$ 
for $|R(T)|=1$ and consequently $|F(T)|=|V|-1$. Next we integrate out $s_x$. 
In terms of the auxiliary random variable $S'$ as specified in \eqref{eq:def-rv-S},
the part of the exponent in \eqref{eq:expectation-Oxy-one-root} 
containing $s_x$ can be rewritten as follows:
\begin{align}
\sum_{i\in V} a_ie^{t_i'} (\epsilon  s_i'+s_x)^2 
= & z_{a,t'} E_{a,t',s'}[(\epsilon S'+s_x)^2] \cr
= & z_{a,t'}\var_{a,t',s'}(\epsilon S'+s_x)+z_{a,t'} E_{a,t',s'}[\epsilon S'+s_x]^2 \cr
= & z_{a,t'}\epsilon^2\var_{a,t',s'}( S')
+z_{a,t'}\left(\epsilon E_{a,t',s'}[ S']+s_x\right)^2 
\end{align}
This yields
\begin{align}
& \int_{\R}ds_x\,
\exp\left[-\frac12\sum_{i\in V} a_ie^{t_i'} (\epsilon  s_i'+s_x)^2 \right] \cr
=& \exp\left[-\frac{z_{a,t'}}{2}\epsilon^2\var_{a,t',s'}( S')\right]
\int_{\R}ds_x\,\exp\left[-\frac{z_{a,t'}}{2}
\left(\epsilon E_{a,t',s'}[ S']+s_x\right)^2 \right] \cr
= & \sqrt{\tfrac{2\pi}{z_{a,t'}}}
\exp\left[-\frac{z_{a,t'}}{2}\epsilon^2\var_{a,t',s'}( S')\right]. 
\end{align}
Inserting this into \eqref{eq:expectation-Oxy-one-root}, we obtain 
the equality claimed in \eqref{eq:limit-as-eps-to-zero}.

The $\epsilon$-dependent integrand in \eqref{eq:limit-as-eps-to-zero} increases
as $\epsilon\downarrow 0$. Hence, by the monotone convergence theorem, we get
the claimed limit.
\end{proof}

\medskip\noindent
\begin{proof}[Proof of Lemma \ref{le:contribution-one-root}]
To prove \eqref{eq:claim-fixed-eps}, we compare two special cases of formula
\eqref{eq:limit-as-eps-to-zero} in Lemma \ref{le:exp-chi},
namely $\chi=1$ with first $a=\pi$ and second $a=\pi_x\delta_x$. 
For $a=\pi_x\delta_x$ we have
$\var_{\pi_x\delta_x,t',s'}( S')=0$, hence
\begin{align}
& \sqrt{\tfrac{2\pi}{z_{\pi,t'}}}e^{-\frac{z_{\pi,t'}}{2}\epsilon^2\var_{\pi,t',s'}( S')}
\prod_{i\in V} e^{-\frac12 \pi_i ( e^{t_i'}+\epsilon^2 e^{-t_i'})} 
\le \sqrt{\tfrac{2\pi}{\pi_xe^{t_x'}}} e^{-\frac12 \pi_x ( e^{t_x'}+\epsilon^2 e^{-t_x'})} \cr
= & \sqrt{\tfrac{2\pi}{z_{\pi_x\delta_x,t'}}}
e^{-\frac{z_{\pi_x\delta_x,t'}}{2}\epsilon^2\var_{\pi_x\delta_x,t',s'}( S')}
e^{-\frac12 \pi_x ( e^{t_x'}+\epsilon^2 e^{-t_x'})} . 
\end{align}
Inserting this in the equality in Lemma \ref{le:exp-chi} yields claim 
\eqref{eq:claim-fixed-eps} as follows 
\begin{align}
e^{-\sum_{i\in V} \eps_i} E_{\mu^\eps}[\obs_{xy}^\pi1_{\{R(T)=\{x\}\}}] 
\le &  e^{-\eps_x} 
E_{\mu^{\eps_x\delta_x}}[\obs_{xy}^\pi1_{\{R(T)=\{x\}\}}]
= e^{-\eps_x} E_{\mu^{\eps_x\delta_x}}[\obs_{xy}^\pi]. 
\end{align}
Note that the event $\{R(T)=\{x\}\}$ holds $\mu^{\eps_x\delta_x}$-almost surely.

To prove the remaining claim \eqref{eq:claim-limit-eps}, we compare three special 
cases of formula \eqref{eq:limit-as-eps-to-zero} in Lemma \ref{le:exp-chi}. \\
{\it Case 1:} $a=\pi$, $\chi=1$; \\
{\it Case 2:} $a=\pi_x\delta_x$, 
$\displaystyle\chi(t')=\sqrt{\tfrac{\pi_xe^{t_x'}}{\sum_{i\in V}\pi_ie^{t_i'}}}
\prod_{i\in V\setminus\{x\}} e^{-\frac12\pi_i e^{t_i'}}
=\sqrt{\tfrac{\pi_xe^{t_x}}{\sum_{i\in V}\pi_ie^{t_i}}}
\prod_{i\in V\setminus\{x\}} e^{-\frac12\epsilon\pi_i e^{t_i}}
$; \\
{\it Case 3:} $a=\pi_x\delta_x$, $\chi=1$. \\
Note that 
\begin{align}
\sqrt{\tfrac{2\pi}{z_{\pi,t'}}} 
= \sqrt{\tfrac{2\pi}{z_{\pi_x\delta_x,t'}}}
\sqrt{\tfrac{\pi_xe^{t_x'}}{\sum_{i\in V}\pi_ie^{t_i'}}}
\le \sqrt{\tfrac{2\pi}{z_{\pi_x\delta_x,t'}}}
.
\end{align}
Consequently, the limits in \eqref{eq:limit-as-eps-to-zero} for the first 
two cases coincide, while the limit in the third case yields an 
upper bound for the other two cases. Hence, 
\begin{align}
0< & \lim_{\epsilon\downarrow 0} E_{\mu^\eps}[\obs_{xy}^\pi1_{\{R(T)=\{x\}\}}]\cr
= & \lim_{\epsilon\downarrow 0} E_{\mu^{\eps_x\delta_x}}\left[
\obs_{xy}^\pi \sqrt{\tfrac{\pi_x e^{t_x}}{\sum_{i\in V}\pi_i e^{t_i}}}
\prod_{i\in V\setminus\{x\}} e^{-\frac12\epsilon\pi_i e^{t_i}} 1_{\{R(T)=\{x\}\}}\right] \cr
\le & \lim_{\epsilon\downarrow 0} 
E_{\mu^{\eps_x\delta_x}}\left[\obs_{xy}^\pi 1_{\{R(T)=\{x\}\}}\right].
\end{align}
Recall that the event $\{R(T)=\{x\}\}$ holds $\mu^{\eps_x\delta_x}$-almost surely.
Consequently, we can drop the indicator function in the last two 
expectations. 

Next, we argue that the last limit is finite. Clearly, from  \eqref{eq:def-obs}, we have
\begin{equation}
\obs_{xy}^\pi\le \min \left\{ \frac{e^{t_x} }{\pi_y},\frac{e^{t_y} }{\pi_x}  \right\}.
\end{equation}
By formula \eqref{eq:expectation-2}
in the appendix, we conclude
\begin{align}
E_{\mu^{\eps_x\delta_x}}\left[\obs_{xy}^\pi \right]
\le &E_{\mu^{\eps_x\delta_x}}\left[ \min \left\{ \frac{e^{t_x} }{\pi_y},\frac{e^{t_y} }{\pi_x}  \right\}\right]
\nonumber \\
\leq & \min \left\{   \frac{1 }{\pi_y}E_{\mu^{\eps_x\delta_x}}[e^{t_x}],  
\frac{1}{\pi_x}E_{\mu^{\eps_x\delta_x}}[e^{t_y}]
 \right\}
= \min \left\{\frac{1}{\pi_x}, \frac{1}{\pi_y} \right\}.
\end{align}
Since the upper bound is independent of $\epsilon$, we have the same 
bound for the limit:
$\lim_{\epsilon\downarrow 0}E_{\mu^{\eps_x\delta_x}}\left[\obs_{xy}^\pi \right]\le \min\{\pi_x^{-1},\pi_y^{-1}\} $.
\end{proof}

The next lemma deals with the lower order corrections coming from forests with 
at least two roots. 

\begin{lemma}[Contribution of at least two roots]
\label{le:contribution-more-than-one-root}
\begin{align}
\lim_{\epsilon\downarrow 0} E_{\mu^\eps}[\obs_{xy}^\pi 1_{\{|R(T)|>1\}}]= 0.
\end{align}
\end{lemma}
\begin{proof}
Let $S$ be a fixed spanning tree of $\G$. We drop the interaction terms 
$\beta_{ij}(B_{ij}-1)\ge 0$ along the edges $(i\sim j)\not\in S\cup\{x\sim\rho\}$. 
This yields
\begin{align}
& E_{\mu^\eps}[\obs_{xy}^\pi 1_{\{|R(T)|>1\}}]
= \sum_{T\in\T} 1_{\{|R(T)|>1\}}\int_{(\R^V)^2} \prod_{j\in V} \frac{dt_jds_je^{-t_j}}{2\pi}
\prod_{(i\sim j)\in E} e^{-\beta_{ij}(B_{ij}(t,s)-1)} \cr
& \cdot\prod_{i\in V} e^{-\eps_i (B_{i\rho}(t,s)-1)} 
\prod_{(i\sim j)\in F(T)} \beta_{ij}e^{t_i+t_j}
\prod_{i\in R(T)}\eps_ie^{t_i}
\,\cdot\, \frac{e^{t_x+t_y}}{\pi_x e^{t_x}+\pi_y e^{t_y}}
1_{ \{ (x\sim\rho)\in T \text{ and }x\connected y\}} \cr
\le & \sum_{T\in\T} 1_{\{|R(T)|>1\}}\int_{(\R^V)^2} \prod_{j\in V} \frac{dt_jds_je^{-t_j}}{2\pi }
\prod_{(i\sim j)\in S} e^{-\beta_{ij}(B_{ij}(t,s)-1)}
\cdot e^{-\eps_x (B_{x\rho}(t,s)-1)} \cr
& \cdot\prod_{(i\sim j)\in F(T)} \beta_{ij}e^{t_i+t_j}
\prod_{i\in R(T)}\eps_ie^{t_i} 
\,\cdot\, \frac{e^{t_x}}{\pi_y}
.
\label{eq:before-changing-to-nabla-s}
\end{align}
In the following, we first change variables to $s_x$ and gradient variables 
$s_{ij}:=s_i-s_j$, $(i\sim j)\in S$, along the spanning tree $S$, where the edges in 
$S$ are oriented in a fixed, but arbitrary way. Since $S$ is
a spanning tree, this is a well defined coordinate change. 
Then we integrate the new  variables out. 
\begin{align}
& \text{r.h.s.\ in \eqref{eq:before-changing-to-nabla-s}} 
= \sum_{T\in\T} 1_{\{|R(T)|>1\}}\int_{\R^V} \prod_{j\in V} \frac{dt_j\, e^{-t_j}}{2\pi }
\int_{\R} ds_x \, e^{-\eps_x (\cosh t_x-1+\frac12 s_x^2 e^{t_x})} \cr
& \cdot \int\limits_{\R^S} \prod_{(i\sim j)\in S} d s_{ij}
\prod_{(i\sim j)\in S} e^{-\beta_{ij}(\cosh(t_i - t_j)-1+\frac12  s_{ij}^2 e^{t_i+t_j})} 
\prod_{(i\sim j)\in F(T)} \beta_{ij}e^{t_i+t_j}
\prod_{i\in R(T)}\eps_ie^{t_i} 
\,\cdot\, \frac{e^{t_x}}{\pi_y} \cr
= & \sum_{T\in\T} 1_{\{|R(T)|>1\}}\int_{\R^V} \prod_{j\in V} \frac{dt_j\, e^{-t_j}}{\sqrt{2\pi}}
\cdot e^{-\eps_x (\cosh t_x-1)} \eps_x^{-\frac12} e^{-\frac12 t_x} \cr
& \cdot \prod_{(i\sim j)\in S} e^{-\beta_{ij}(\cosh(t_i - t_j)-1)}
\beta_{ij}^{-\frac12} e^{-\frac12(t_i+t_j)}
\prod_{(i\sim j)\in F(T)} \beta_{ij}e^{t_i+t_j}
\prod_{i\in R(T)}\eps_ie^{t_i} 
\,\cdot\, \frac{e^{t_x}}{\pi_y}
.
\label{eq:after-integrating-s-out}
\end{align}
Next, we set 
\begin{align}
\label{eq:def-change-of-variables-t}
t_x':=t_x-\log\epsilon,\quad \tau_i:=t_i-t_x'-\log\epsilon
\end{align} 
for $i\in V$. In particular, $\tau_x=0$; thus, we use $t_x'$ and $\tau_i$, $i\in V\setminus\{x\}$,
as new integration variables. Note that this substitution is 
different from the one in the proof of Lemma \ref{le:contribution-one-root}. 
Heuristically speaking, the reason is that in the case of $|R(T)|>1$ most of the mass 
of the $t_x$-integral is located near $t_x\approx+\log\epsilon$, while in the 
case of one root $R(T)=\{x\}$ the mass is essentially located near
$t_x\approx-\log\epsilon$. To do the power counting for $\epsilon$ and 
$e^{t_x'}$ in the following calculation, we use 
\begin{align}
|F(T)|+|R(T)|=|V_\rho|-1=|V|=|S|+1. 
\end{align}
We obtain 
\begin{align}
\eqref{eq:after-integrating-s-out} 
= &  \sum_{T\in\T} 1_{\{|R(T)|>1\}}
\int_{\R} \frac{dt_x'\, e^{-t_x'}}{\epsilon\sqrt{2\pi}}
e^{-\epsilon\pi_x [\frac12(\epsilon e^{t_x'}+\epsilon^{-1} e^{-t_x'})-1]} 
(\epsilon\pi_x)^{-\frac12} \epsilon^{-\frac12}e^{-\frac12 t_x'} \cr
& \cdot \int_{\R^{V\setminus\{x\}}} \prod_{j\in V\setminus\{x\}} 
\frac{d\tau_j\, e^{-t_x'-\tau_j}}{\epsilon \sqrt{2\pi}}
\prod_{(i\sim j)\in S} [ e^{-\beta_{ij}(\cosh(\tau_i - \tau_j)-1)}
\beta_{ij}^{-\frac12} \epsilon^{-1}e^{-t_x'}e^{-\frac12(\tau_i+\tau_j)} ]\cr
& \prod_{(i\sim j)\in F(T)} \beta_{ij}\epsilon^2 e^{2t_x'+\tau_i+\tau_j}
\prod_{i\in R(T)}\epsilon^2\pi_ie^{t_x'+\tau_i} 
\,\cdot\, \frac{\epsilon e^{t_x'}}{\pi_y} \cr
= & \epsilon e^{\epsilon\pi_x}\sum_{T\in\T} 1_{\{|R(T)|>1\}}
\int_{\R} \frac{dt_x'}{\sqrt{2\pi}}e^{(\frac32-|R(T)|)t_x'}
e^{-\frac{\pi_x}{2}(\epsilon^2 e^{t_x'}+e^{-t_x'})} \pi_x^{-\frac12} \cr
& \cdot \int_{\R^{V\setminus\{x\}}} \prod_{j\in V\setminus\{x\}} 
\frac{d\tau_j\, e^{-\tau_j}}{\sqrt{2\pi}}
\prod_{(i\sim j)\in S} [ e^{-\beta_{ij}(\cosh(\tau_i - \tau_j)-1)}
\beta_{ij}^{-\frac12} e^{-\frac12(\tau_i+\tau_j)} ]\cr
& \prod_{(i\sim j)\in F(T)} \beta_{ij} e^{\tau_i+\tau_j}
\prod_{i\in R(T)}\pi_ie^{\tau_i} \,\cdot\, \frac{1}{\pi_y} .
\label{eq:after-changing-to-grad-t}
\end{align}
Next, we drop the term $e^{-\frac{\pi_x}{2}\epsilon^2 e^{t_x'}}\le 1$. For 
any $T\in\T$ with $|R(T)|\ge 2$, we obtain 
\begin{align}
\int_{\R} \frac{dt_x'}{\sqrt{2\pi}}e^{(\frac32-|R(T)|)t_x'}
e^{-\frac{\pi_x}{2}(\epsilon^2 e^{t_x'}+e^{-t_x'})}
\le & \int_{\R} \frac{dt_x'}{\sqrt{2\pi}}e^{(\frac32-|R(T)|)t_x'}
e^{-\frac{\pi_x}{2}e^{-t_x'}} \cr
=: & \ceins(\pi,|R(T)|)<\infty.
\end{align}
Note that in this integral, the integrand decays superexponentially for 
$t_x'$ near $-\infty$ and exponentially for $t_x'$ near $+\infty$. 
Thus, we get 
\begin{align}
\eqref{eq:after-changing-to-grad-t} \le & 
\epsilon e^{\epsilon\pi_x} \sum_{T\in\T} 1_{\{|R(T)|>1\}} \ceins(\pi,|R(T)|)
\pi_x^{-\frac12} 
\int_{\R^{V\setminus\{x\}}} \prod_{j\in V\setminus\{x\}} 
\frac{d\tau_j\, e^{-\tau_j}}{\sqrt{2\pi}}\cr
& \prod_{(i\sim j)\in S} [ e^{-\beta_{ij}(\cosh(\tau_i - \tau_j)-1)}
\beta_{ij}^{-\frac12} e^{-\frac12(\tau_i+\tau_j)} ]
\prod_{(i\sim j)\in F(T)} \beta_{ij} e^{\tau_i+\tau_j}
\prod_{i\in R(T)}\pi_ie^{\tau_i} \,\cdot\, \frac{1}{\pi_y} \cr
=: & \epsilon e^{\epsilon\pi_x} \czwei(\pi,\beta,\G).
\end{align}
Note that $\czwei(\pi,\beta,\G)<\infty$ because the product over 
$ e^{-\beta_{ij}[\cosh(\tau_i-\tau_j)-1]}$ decays superexponentially fast (recall that $\tau_x=0$).
Summarizing, we get 
\begin{align}
0\le E_{\mu^\eps}[\obs_{xy}^\pi 1_{\{|R(T)|>1\}}] \le \epsilon e^{\epsilon\pi_x} 
\czwei(\pi,\beta,\G)
\konv{\epsilon\downarrow 0}0.  
\end{align}
\end{proof}

The main theorem \ref{thm:comparison} is now proved by a combination of Lemmas 
\ref{le:contribution-one-root} and \ref{le:contribution-more-than-one-root}:

\medskip\noindent
\begin{proof}[Proof of Theorem \ref{thm:comparison}]
From \eqref{eq:obs-as-sum}, we get 
\begin{align}
\label{eq:exp-obs-as-sum}
E_{\mu^\eps}[\obs_{xy}^\pi] 
=E_{\mu^\eps}[\obs_{xy}^\pi1_{\{R(T)=\{x\}\}}]+E_{\mu^\eps}[\obs_{xy}^\pi1_{\{|R(T)|>1\}}].
\end{align}
Combining this with Lemma \ref{le:contribution-one-root} and 
Lemma \ref{le:contribution-more-than-one-root} yields
\begin{align}
0< & \lim_{\epsilon\downarrow 0} E_{\mu^\eps}[\obs_{xy}^\pi]
= \lim_{\epsilon\downarrow 0} E_{\mu^{\eps_x\delta_x}}\left[
\obs_{xy}^\pi \sqrt{\tfrac{\pi_x e^{t_x}}{\sum_{i\in V}\pi_i e^{t_i}}}
\prod_{i\in V\setminus\{x\}} e^{-\frac12\epsilon\pi_i e^{t_i}} \right] \cr
\le & \lim_{\epsilon\downarrow 0} 
E_{\mu^{\eps_x\delta_x}}\left[\obs_{xy}^\pi\right]<\infty. 
\label{eq:claim-x-y}
\end{align}
Using \eqref{eq:G-xy}, we obtain 
\begin{align}
\label{eq:expectation-G-xy}
\epsilon E_{\mu^\eps}[G_{xy}^\eps] = E_{\mu^\eps}[\obs_{xy}^\pi+\obs_{yx}^\pi]
.
\end{align}
Applying \eqref{eq:claim-x-y} twice, as it is and with $x$ and $y$ 
interchanged, the claim follows.
\end{proof}

Finally, specializing the theorem down to ladder graphs, we transfer
our results from \cite{disertori-merkl-rolles2014} concerning exponential 
decay of weights in the case of pinning at one point to the case of 
uniform pinning (or more general pinning): 

\medskip\noindent
\begin{proof}[Proof of Corollary \ref{cor:exp-decay}]
Recall that $\cfuenf=\min \{\pi_x,\pi_y\}^{-1}$. We estimate
\begin{align}
\obs_{xy}^\pi \le \frac{e^{t_x+t_y}}{\pi_x e^{t_x}+\pi_y e^{t_y}}
\le \cfuenf \min\{ e^{t_x},e^{t_y}\}
\le  \cfuenf e^{t_x}e^{\frac14(t_y-t_x)}.
\end{align}
By Lemma \ref{le:independence-tx-gradients}, with respect to $\mu_L^{\eps_x\delta_x}$, 
the random variables $e^{t_x}$ and $e^{\frac14(t_y-t_x)}$ are stochastically independent 
and the distribution of $e^{\frac14(t_y-t_x)}$ is independent of $\eps_x$. Furthermore, 
$E_{\mu_L^{\eps_x\delta_x}}[e^{t_x}]=1$. Thus, for every $\epsilon>0$, we have 
\begin{align}
E_{\mu_L^{\eps_x\delta_x}}[\obs_{xy}^\pi]
\le & \cfuenf E_{\mu_L^{\eps_x\delta_x}}\big[e^{t_x}e^{\frac14(t_y-t_x)}\big]
=\cfuenf E_{\mu_L^{\eps_x\delta_x}}\big[e^{t_x}\big] E_{\mu_L^{\eps_x\delta_x}}\big[e^{\frac14(t_y-t_x)}\big] \cr
=& \cfuenf E_{\mu_L^{\eps_x\delta_x}}\big[e^{\frac14(t_y-t_x)}\big] 
= \cfuenf E_{\mu_L^{\delta_x}}\big[e^{\frac14(t_y-t_x)}\big];
\label{eq:est-exp-obs}
\end{align}
in the last expectation we replaced $\eps_x$ by $1$.

Let $z$ denote the copy of $x$ at the level of $y$, i.e.\ if 
$x=(n,v)$ and $y=(m,w)$, then $z:=(m,v)$. Using the Cauchy Schwarz inequality, we obtain
\begin{align}
E_{\mu_L^{\delta_x}}\big[e^{\frac14(t_y-t_x)}\big] 
= E_{\mu_L^{\delta_x}}\big[e^{\frac14(t_y-t_z)}e^{\frac14(t_z-t_x)}\big]
\le E_{\mu_L^{\delta_x}}\big[e^{\frac12(t_y-t_z)}\big]^{\frac12} E_{\mu_L^{\delta_x}}\big[e^{\frac12(t_z-t_x)}\big]^{\frac12} .
\end{align}
By Theorem 2.1 in \cite{disertori-merkl-rolles2014}, there exist 
constants $\csechs,\csieben>0$ depending only on $G_0$ and $\beta$ such that
\begin{align}
E_{\mu_L^{\delta_x}}\big[e^{\frac12(t_z-t_x)}\big]
\le \csechs e^{-\csieben|z-x|}=\csechs e^{-\csieben|y-x|}. 
\end{align}
For the points $y$ and $z$ on the same level, estimate (7.6) from 
\cite{disertori-merkl-rolles2014} states 
\begin{align}
E_{\mu_L^{\delta_x}}\big[e^{\frac12(t_y-t_z)}\big]\le\cacht
\label{eq:expectation-same-level}
\end{align}
with a constant $\cacht$ depending only on $G_0$ and $\beta$. 
Summarizing, \eqref{eq:est-exp-obs}--\eqref{eq:expectation-same-level} yield 
\begin{align}
E_{\mu_L^{\eps_x\delta_x}}[\obs_{xy}^\pi] 
\le\cfuenf(\csechs\cacht)^{\frac12} e^{-\frac12\csieben|y-x|}
=:\frac{\cdrei\cfuenf}{2} e^{-\cvier|y-x|} 
\end{align}
with constants $\cdrei(G_0,\beta),\cvier(G_0,\beta)>0$ uniformly in $\epsilon>0$. 
This shows 
\begin{align}
\lim_{\epsilon\downarrow 0} E_{\mu_L^{\eps_x\delta_x}}[\obs_{xy}^\pi] 
\le \frac{\cdrei\cfuenf}{2} e^{-\cvier|y-x|} . 
\end{align}
Interchanging the roles of $x$ and $y$, we get the same upper bound
for $\lim_{\epsilon\downarrow 0} E_{\mu_L^{\eps_y\delta_y}}[\obs_{yx}^\pi]$. 
An application of Theorem \ref{thm:comparison} yields the claim. 
\end{proof}

\begin{appendix}
\section{Appendix: Product structure of the model with single pinning}
When transforming the model $H^{2|2}$ with pinning at one point to gradient variables, it 
exhibits a certain 
product structure coming from scaling symmetry. This is made precise in the following lemma. 

\begin{lemma}
\label{le:independence-tx-gradients}
With respect to $\mu^{\eps_x\delta_x}$, the random pair $(t_x,s_x)$ is independent of the 
random vector consisting of the (rescaled) gradient variables 
\begin{align}
(t_i':=t_i-t_x,s_i':=(s_i-s_x)e^{t_x})_{i\in V\setminus\{x\}}.
\end{align}
The joint distribution of $(t_x,s_x)$ with respect to $\mu^{\eps_x\delta_x}$ has the density 
\begin{align}
\label{eq:density-tx-sx}
\frac{\eps_x}{2\pi} \exp\left[-\eps_x\left(\cosh t_x-1+\tfrac12 s_x^2e^{t_x}\right)\right],
\end{align}
independently of the graph $\G$. In particular, 
\begin{align}
\label{eq:expectation-1}
E_{\mu^{\eps_x\delta_x}}[e^{t_x}]=1. 
\end{align}
The joint distribution of 
$(t_i',s_i')_{i\in V\setminus\{x\}}$ does not dependent on $\eps_x$. 
\end{lemma}
\begin{proof}
Recall the definition of $\mu^\eps$ given in \eqref{eq:mu-eps}. In the 
special case $\eps=\eps_x\delta_x$, the random tree $T$ contains 
$\mu^{\eps_x\delta_x}$-almost surely the edge $x\sim\rho$, but no other edge 
of the type $i\sim\rho$, $i\neq x$. Hence, we get
\begin{align}
& \mu^{\eps_x\delta_x} (dt\, ds\, dT) \cr
= & \prod_{j\in V} \frac{dt_jds_je^{-t_j}}{2\pi }\, dT\, 
e^{-\eps_x(B_{x\rho}(t,s)-1)}\eps_xe^{t_x}
\prod_{(i\sim j)\in E} e^{-\beta_{ij}(B_{ij}(t,s)-1)}
\prod_{(i\sim j)\in F(T)} \beta_{ij}e^{t_i+t_j}.
\end{align}
Let $\nu^{\eps_x\delta_x}$ denote the joint distribution of 
$(t_x,s_x,(t_i',s_i')_{i\neq x})$. 
We set $t_x':=0$ and $s_x':=0$. Note that 
$(s_i-s_j)^2e^{t_i+t_j}=(s_i'-s_j')^2 e^{t_i'+t_j'}$. 
Changing variables accordingly and denoting the set of spanning trees
of the graph $\G$ by $\T_\G$, we obtain 
\begin{align}
& \nu^{\eps_x\delta_x} (dt_x\, ds_x\, dt'\, ds') 
= \frac{dt_xds_xe^{-t_x}}{2\pi}
\prod_{j\in V\setminus\{x\}} \frac{dt_j'ds_j'e^{-2t_x-t_j'}}{2\pi} \cr
& e^{-\eps_x(\cosh(t_x)-1+\frac12 s_x^2 e^{t_x})}
\eps_xe^{t_x}
\prod_{(i\sim j)\in E} e^{-\beta_{ij}[B_{ij}(t',s')-1]} 
\sum_{T\in\T_\G}\prod_{(i\sim j)\in T} \beta_{ij}e^{2t_x+t_i'+t_j'}\cr
= & \frac{dt_xds_x\eps_x}{2\pi}
e^{-\eps_x(\cosh(t_x)-1+\frac12 s_x^2 e^{t_x})}\cr
& \cdot \prod_{j\in V\setminus\{x\}} \frac{dt_j'ds_j'e^{-t_j'}}{2\pi} 
\prod_{(i\sim j)\in E} e^{-\beta_{ij}[B_{ij}(t',s')-1]} 
\sum_{T\in\T_\G}\prod_{(i\sim j)\in T} \beta_{ij}e^{t_i'+t_j'}.
\label{eq:independence}
\end{align}
In the special case of the graph $\G$ consisting of only one point $x$, 
i.e.\ $V=\{x\}$ and $E=\emptyset$, the measure $\nu^{\eps_x\delta_x}$ has the density 
given in \eqref{eq:density-tx-sx}. Since $\nu^{\eps_x\delta_x}$ is a probability measure, 
the density in \eqref{eq:density-tx-sx} is normalized to have total mass one. 
Consequently, given the product structure in \eqref{eq:independence}, for 
a general graph $\G$, the random vectors $(t_x,s_x)$ and $(t',s')$ 
are independent with the claimed first marginal and the second marginal 
not depending on $\eps_x$.
Finally, we calculate 
\begin{align}
E_{\mu^{\eps_x\delta_x}}[e^{t_x}]
\stackrel{\phantom{\text{(by symmetry)}}}{=}  & 
\frac{\eps_x}{2\pi}\int_{\R^2} e^{t_x}\exp[-\eps_x(\cosh t_x-1+\tfrac12 s_x^2e^{t_x})]
\, ds_x dt_x 
\cr
\stackrel{\phantom{\text{(by symmetry)}}}{=}  &\sqrt{\tfrac{\eps_x}{2\pi}}
\int_{\R} e^{\frac{t_x}{2}}\exp[-\eps_x(\cosh t_x-1)]\, dt_x 
\cr
\stackrel{\text{(by symmetry)}}{=} & \sqrt{\tfrac{\eps_x}{2\pi}}
\int_{\R} e^{-\frac{t_x}{2}}\exp[-\eps_x(\cosh t_x-1)]\, dt_x \cr
\stackrel{\phantom{\text{(by symmetry)}}}{=}  &  E_{\mu^{\eps_x\delta_x}}[1] 
=1.
\end{align}
\end{proof}

Using supersymmetry,  identity \eqref{eq:expectation-1} can be generalized as follows.
\begin{lemma}[Formula (B.3) in \cite{disertori-spencer-zirnbauer2010}]
 For any $y\in V$ and any choice of $\eps$ we have 
\begin{align}
\label{eq:expectation-2}
E_{\mu^{\eps}}[e^{t_y}]=1. 
\end{align}
\end{lemma}

\end{appendix}

\bibliographystyle{alpha}

\end{document}